\documentclass{amsart}

\usepackage{amsmath,amssymb,amsthm,amscd}
\newtheorem{theorem}{Theorem}
\newtheorem{lemma}[theorem]{Lemma}
\newtheorem{proposition}[theorem]{Proposition}
\newtheorem{corollary}[theorem]{Corollary}
\newtheorem{remark}[theorem]{Remark}
\newtheorem{example}[theorem]{Example}

\begin{document}

\title[Weak dimension of FP-injective modules]{Weak dimension of FP-injective modules over chain rings}
\author{Fran\c{c}ois Couchot}
\address{Universit\'e de Caen Basse-Normandie, CNRS UMR
  6139 LMNO,
F-14032 Caen, France}
\email{francois.couchot@unicaen.fr} 


\begin{abstract} It is proven that the weak dimension of each FP-injective module over a chain ring which is either Archimedean or not semicoherent is less or equal to $2$. This implies that the projective dimension of any countably generated FP-injective module over an Archimedean chain ring is less or equal to $3$. 
\end{abstract}

\maketitle

By \cite[Theorem 1]{Cou12}, for any module $G$ over a commutative arithmetical ring $R$ the weak dimension of $G$ is $0,\ 1,\ 2$ or $\infty$. In this paper we consider the weak dimension of almost FP-injective modules over a chain ring. This class of modules contains the one of FP-injective modules and these two classes coincide if and only if the ring is coherent. If $G$ is an almost FP-injective module over a chain ring $R$ then its weak dimension is possibly infinite only if $R$ is semicoherent and not coherent. In the other cases the weak dimension of $G$ is at most $2$. Moreover this dimension is not equal to $1$ if $R$ is not an integral domain. Theorem \ref{T:main} summarizes main results of this paper.

We complete this short paper by considering almost FP-injective modules over local fqp-rings. This class of rings was introduced in \cite{AbJaKa11} by Abuhlail, Jarrar and Kabbaj. It contains the one of arithmetical rings. It is shown the weak dimension of $G$ is infinite if $G$ is an almost FP-injective module over a local fqp-ring which is not a chain ring (see Theorem \ref{T:main1}).

All rings in this paper are unitary and commutative.
A  ring $R$ is said to be a {\bf chain ring}\footnote{we prefer ``chain ring '' to ``valuation ring'' to avoid confusion with ``Manis valuation ring''.}  if its lattice of ideals is totally ordered by inclusion. Chain rings are also called valuation rings (see \cite{FuSa85}).
If $M$ is an $R$-module, we denote by $\mathrm{w.d.}(M)$ its {\bf weak dimension} and $\mathrm{p.d.}(M)$ its {\bf projective dimension}. Recall that $\mathrm{w.d.}(M)\leq n$ if $\mathrm{Tor}^R_{n+1}(M,N)=0$ for each $R$-module $N$. For any ring $R$, its {\bf global weak dimension} $\mathrm{w.gl.d}(R)$ is the supremum of $\mathrm{w.d.}(M)$ where $M$ ranges over all (finitely presented cyclic) $R$-modules. Its {\bf finitistic weak dimension}
$\mathrm{f.w.d.}(R)$ is the supremum of $\mathrm{w.d.}(M)$ where $M$ ranges over all  $R$-modules of finite weak dimension. 

A ring is called {\bf coherent} if all its finitely generated ideals are finitely presented. As  in \cite{Mat85}, a ring $R$ is said to be \textbf{semicoherent} if $\mathrm{Hom}_R(E,F)$ is a submodule of a flat $R$-module for any pair of injective $R$-modules $E,\ F$. A ring $R$ is said to be \textbf{IF} ({\bf semi-regular} in \cite{Mat85}) if each injective $R$-module is flat. If $R$ is a chain ring, we denote by $P$ its maximal ideal, $Z$ its  subset of zerodivisors which is a prime ideal and $Q(=R_Z)$ its fraction ring. If $x$ is an element of a module $M$ over a ring $R$, we denote by $(0:x)$ the annihilator ideal of $x$ and by $\mathrm{E}(M)$ the injective hull of $M$.


Some preliminary results are needed to prove Theorem \ref{T:main} which is the main result of this paper.
 
\begin{proposition} \cite[Proposition 4] {Cou12}.
\label{P:ValSemiCoh} Let $R$ be a chain ring.  The following conditions are equivalent:
\begin{enumerate}
\item $R$ is semicoherent; 
\item $Q$ is an IF-ring;
\item $Q$ is coherent;
\item either $Z=0$ or $Z$ is not flat;
\item $\mathrm{E}(Q/Qa)$ is flat for each nonzero element $a$ of $Z$;
\item there exists $0\ne a\in Z$ such that $(0:a)$ is finitely generated over $Q$.
\end{enumerate} 
\end{proposition}

An exact sequence of $R$-modules $0 \rightarrow F \rightarrow E \rightarrow G \rightarrow 0$  is {\bf pure}
if it remains exact when tensoring it with any $R$-module. Then, we say that $F$ is a \textbf{pure} submodule of $E$. When $E$ is flat, it is well known that $G$ is flat if and only if $F$ is a pure submodule of $E$. An $R$-module $E$ is {\bf FP-injective}  if 
$\hbox{Ext}_R^1 (F, E) = 0$  for any finitely presented $R$-module $F$.  We recall that a module $E$ is FP-injective if and only if it is a pure submodule
of every overmodule. We define a module $G$ to be {\bf almost FP-injective} if there exist a FP-injective module $E$ and a pure submodule $D$ such that $G\cong E/D$. By \cite[35.9]{Wis91} the following theorem holds:

\begin{theorem}
\label{T:coh} A ring $R$ is coherent if and only if each almost FP-injective $R$-module is FP-injective.
\end{theorem}

\begin{proposition}
\label{P:flat} For any ring $R$ the following conditions are equivalent:
\begin{enumerate}
\item $R$ is self FP-injective;
\item each flat $R$-module is almost FP-injective.
\end{enumerate}
\end{proposition}
\begin{proof}
$(1)\Rightarrow (2)$. Each flat module $G$ is of the form $F/K$ where $F$ is free and $K$ is a pure submodule of $F$. Since $F$ is FP-injective then $G$ is almost FP-injective.

$(2)\Rightarrow (1)$. We use the fact that $R$ is projective to show that $R$ is a direct summand of a FP-injective module.
\end{proof}

\begin{proposition}\label{P:pure}
Let $R$ be a ring and $0\rightarrow L\xrightarrow{v} M\rightarrow N\rightarrow 0$ a pure exact sequence of $R$-modules. Then $\mathrm{w.d.}(M)=\max (\mathrm{w.d.}(L),\mathrm{w.d.}(N))$.
\end{proposition}
\begin{proof} Let $m=\mathrm{w.d.}(M)$ and let $G$ be an $R$-module. We consider the following flat resolution of $G$:
\[F_p\xrightarrow{u_p} F_{p-1}\dots F_2\xrightarrow{u_2} F_1\xrightarrow{u_1} F_0\rightarrow G\rightarrow 0.\]
For each positive integer $p$ let $G_p$ be the image of $u_p$. We have the following exact sequence:
\[\mathrm{Tor}_1^R(G_m,M)\rightarrow \mathrm{Tor}_1^R(G_m,N)\rightarrow G_m\otimes_RL\xrightarrow{id_{G_m}\otimes v} G_m\otimes_RM.\]
Since $\mathrm{Tor}_1^R(G_m,M)\cong\mathrm{Tor}_{m+1}^R(G,M)=0$ and $id_{G_m}\otimes v$ is a monomorphism we deduce that $\mathrm{Tor}_{m+1}^R(G,N)\cong\mathrm{Tor}_{1}^R(G_m,N)=0$. So, $\mathrm{w.d.}(N)\leq m$. Now it is easy to show first, that $\mathrm{w.d.}(L)\leq m$ too, and then, if $q=\max (\mathrm{w.d.}(L),\mathrm{w.d.}(N))$, that $\mathrm{Tor}_{q+1}^R(G,M)=0$ for each $R$-module $G$. Hence $m=q$.
\end{proof}

\begin{proposition}\label{P:R/Z}
Let $R$ be a chain ring. Then:
\begin{enumerate}
\item $\mathrm{w.d.}(R/Z)=\infty$ if $R$ is semicoherent and $Z\ne 0$;
\item $\mathrm{w.d.}(R/Z)=1$ if $R$ is not semicoherent.
\end{enumerate}
\end{proposition}
\begin{proof}
We consider the following exact sequence: $0\rightarrow R/Z\rightarrow Q/Z\rightarrow Q/R\rightarrow 0$. If $R\ne Q$ then $\mathrm{w.d.}(Q/R)=1$.

$(1)$. By \cite[Propositions 8 and 14]{Couch03} applied to $Q$ and Proposition \ref{P:ValSemiCoh} $\mathrm{w.d.}(Q/Z)=\infty$. By using the previous exact sequence when $R\ne Q$ we get $\mathrm{w.d.}(R/Z)=\infty$.

$(2)$. By Proposition \ref{P:ValSemiCoh} $\mathrm{w.d.}(Q/Z)=1$. When $R\ne Q$ we conclude by using the previous exact sequence.
\end{proof}

\begin{proposition}\label{P:H}
Let $R$ be a chain ring. If $R$ is IF, let $H=\mathrm{E}(R/rR)$ where $0\ne r\in P$. If $R$ is not IF, let $x\in\mathrm{E}(R/Z)$ such that $Z=(0:x)$ and $H=\mathrm{E}(R/Z)/Rx$. Then:
\begin{enumerate}
\item  $H$ is FP-injective; 
\item for each $0\ne r\in R$ there exists $h\in H$ such that $(0:h)=Rr$;
\item $\mathrm{w.d.}(H)=\infty$ if $R$ is semicoherent and $R\ne Q$;
\item $\mathrm{w.d.}(H)=2$ if $R$ is not semicoherent.
\end{enumerate}
\end{proposition}
\begin{proof}
$(1)$. When $R\ne Q$, $H$ is FP-injective by \cite[Proposition 6]{Couch03}. If $R=Q$ is not IF then $H\cong\mathrm{E}(R/rR)$ for each $0\ne r\in P$ by \cite[Proposition 14]{Couch03}. 

$(2)$. Let $0\ne r\in R$. Then $(0:r)\subseteq Z=(0:x)$. From the injectivity of $\mathrm{E}(R/Z)$ we deduce that there exists $y\in\mathrm{E}(R/Z)$ such that $x=ry$. Now, if we put $h=y+Rx$ it is easy to check that $(0:h)=(Rx:y)=Rr$.

$(3)$ and $(4)$. By \cite[Proposition 8]{Couch03} $\mathrm{E}(R/Z)$ is flat. We use the exact sequence $0\rightarrow R/Z\rightarrow\mathrm{E}(R/Z)\rightarrow H\rightarrow 0$ and Proposition \ref{P:R/Z} to conclude.
\end{proof}

The following proposition is a slight modification of \cite[Proposition 9]{Couch03}.
\begin{proposition}\label{P:inj}
Let $R$ be a chain ring and $G$ an injective module. Then there exists a pure exact sequence:
$0\rightarrow K\rightarrow I\rightarrow G\rightarrow 0$, such
that $I$ is a direct sum of submodules isomorphic to $R$ or $H.$ 
\end{proposition}
\begin{proof}
There exist a set $\Lambda$ and an epimorphism
$\varphi:L=R^{\Lambda}\rightarrow G$. We put
$\Delta=\mathrm{Hom}_R(H,G)$ and $\rho:H^{(\Delta)}\rightarrow G$ the morphism
defined by the elements of $\Delta.$ Thus $\psi$ and $\rho$ induce an
epimorphism $\phi:I=R^{(\Lambda)}\oplus H^{(\Delta)}\rightarrow G.$ Since, for
every $r\in P,r\ne 0,$ each morphism $g:R/Rr\rightarrow G$ can be extended to
$H\rightarrow G,$ we deduce that $K=\ker \phi$ is a pure
submodule of $I.$
\end{proof}

\begin{lemma}
\label{L:aFP} Let $G$ be an almost FP-injective module over a chain ring $R$. Then for any $x\in G$ and $a\in R$ such that $(0:a)\subset (0:x)$ there exists $y\in G$ such that $x=ay$.
\end{lemma}
\begin{proof}
We have $G=E/D$ where $E$ is a FP-injective module and $D$ a pure submodule. Let $e\in E$ such that $x=e+D$. Let $b\in (0:x)\setminus (0:a)$. Then $be\in D$. So, we have $b(e-d)=0$ for some $d\in D$. Whence $(0:a)\subset Rb\subseteq (0:e-d)$. From the FP-injectivity of $E$ we deduce that $e-d=az$ for some $z\in E$. Hence $x=ay$ where $y=z+D$.
\end{proof}

Let $M$ be a non-zero module over a  ring $R$. We set:
\[M_{\sharp}=\{s\in R\mid \exists 0\ne x\in M\ \mathrm{such\ that}\ sx=0\}\quad\mathrm{and}\quad M^{\sharp}=\{s\in R\mid sM\subset M\}.\]
Then $R\setminus M_{\sharp}$ and $R\setminus M^{\sharp}$ are multiplicative subsets of $R$. 
If $M$ is a module over a chain ring $R$ then $M_{\sharp}$ and $M^{\sharp}$ are prime ideals called respectively the {\bf bottom prime ideal} and the {\bf top prime ideal} associated with $M$.

\begin{proposition}
\label{P:sharp} Let $G$ be a module over a chain ring $R$. Then:
\begin{enumerate}
\item $G_{\sharp}\subseteq Z$ if $G$ is flat;
\item $G^{\sharp}\subseteq Z$ if $G$ is almost FP-injective;
\item $G^{\sharp}\subseteq Z\cap G_{\sharp}$ if $G$ is FP-injective. So, $G$ is a module over $R_{G_{\sharp}}$;
\item $G$ is flat and FP-injective if $G$ is almost FP-injective and $G_{\sharp}\cup G^{\sharp}\subset Z$. In this case $G^{\sharp}\subseteq G_{\sharp}$.
\end{enumerate}
\end{proposition}
\begin{proof}
$(1)$. Let $a\in G_{\sharp}$. There exists $0\ne x\in G$ such that $ax=0$. The flatness of $G$ implies that $x\in (0:a)G$. So, $(0:a)\ne 0$ and $a\in Z$.

$(2)$. Let $s\in R\setminus Z$. Then for each $x\in G$, $0=(0:s)\subseteq (0:x)$. If $G$ is FP-injective then $x=sy$ for some $y\in G$. If $G$ is almost FP-injective then it is factor of a FP-injective module, so,  the multiplication by $s$ in $G$ is surjective. 

$(3)$. Let $a\in R\setminus G_{\sharp}$ and $x\in G$. Let $b\in (0:a)$. Then $abx=0$, whence $bx=0$. So, $(0:a)\subseteq (0:x)$. It follows that $x=ay$ for some $y\in G$ since $G$ is FP-injective. Hence $a\notin G^{\sharp}$.

$(4)$. Let $0\rightarrow X\rightarrow E\rightarrow G\rightarrow 0$ be a pure exact sequence where $E$ is FP-injective, and $L=G^{\sharp}\cup G_{\sharp}$. Then $0\rightarrow X_L\rightarrow E_L\rightarrow G\rightarrow 0$ is a pure exact sequence and by \cite[Theorem 3]{Cou06} $E_L$ is FP-injective. By \cite[Theorem 11]{Couch03} $R_L$ is an IF-ring. Hence $G$ is FP-injective and flat.
\end{proof}

\begin{example}
Let $D$ be a valuation domain. Assume that $D$ contains a non-zero prime ideal $L\ne P$ and let $0\ne a\in L$. By \cite[Example 11 and Corollary 9]{Cou12} $R=D/aD_L$ is semicoherent and not coherent. Since $R$ is not self FP-injective then $\mathrm{E}_R(R/P)$ is not flat by \cite[Proposition 2.4]{Cou82}.
\end{example}

\begin{example}
\label{E:example} Let $R$ be a semicoherent chain ring which is not coherent and $G$ an FP-injective module which is not flat. Then $G_Z$ is almost FP-injective over $R$ but not over $Q$, and $(G_Z)_{\sharp}\cup(G_Z)^{\sharp}=Z$.
\end{example}
\begin{proof} Let $G'$ be the kernel of the canonical homomorphism $G\rightarrow G_Z$. By Proposition \ref{P:sharp} the multiplication in $G$ and $G/G'$ by any $s\in R\setminus Z$ is surjective. So, $G_Z\cong G/G'$. For each $s\in P\setminus Z$ we put $G_{(s)}=G/(0:_Gs)$. Then $G_{(s)}$ is FP-injective because it is isomorphic to $G$. Since $G'=\cup_{s\in P\setminus Z}(0:_Gs)$ then $G_Z\cong\varinjlim_{s\in P\setminus Z} G_{(s)}$. By \cite[33.9(2)]{Wis91} $G_Z$ is factor of the FP-injective module $\oplus_{s\in P\setminus Z}G_{(s)}$ modulo a pure submodule. Hence $G_Z$ is almost FP-injective over $R$. Since $Q$ is IF and $G_Z$ is not FP-injective by \cite[Theorem 3]{Cou06}, from  Theorem \ref{T:coh}  we deduce that $G_Z$ is not almost FP-injective over $Q$. We complete the proof by using Proposition \ref{P:sharp}(4).
\end{proof}

The following lemma is a consequence of \cite[Lemma 3]{Gil71} and \cite[Proposition 1.3]{KlLe69}.

\begin{lemma}
\label{L:ann} Let $R$ be a chain ring for which $Z=P$, and $A$ an ideal of $R$. Then $A\subset (0:(0:A))$ if and only if $P$ is faithful and there exists $t\in R$ such that $A=tP$ and $(0:(0:A))=tR$. 
\end{lemma}

\begin{proposition}\label{P:non1}
Let $R$ be a chain ring for which $Z\ne 0$. Then $\mathrm{w.d.}(G)\ne 1$ for any almost FP-injective $R$-module $G$.
\end{proposition}
\begin{proof}
By way of contradiction assume there exists an almost FP-injective $R$-module $G$ with $\mathrm{w.d.}(G)=1$. There exists an exact sequence $0\rightarrow K\rightarrow F\rightarrow G\rightarrow 0$ which is not pure, where $F$ is free and $K$ is flat. 

First we assume that $R=Q$, whence $P=Z$. So there exist $b\in R$ and $x\in F$ such that $bx\in K\setminus bK$. We put $B=(K:x)$. From $b\in B$ we deduce that $(0:B)\subseteq (0:b)$. We investigate the following cases:

{\bf 1}. $(0:B)\subset (0:b)$. 

Let $a\in (0:b)\setminus (0:B)$. Then $(0:a)\subset (0:(0:B))$. If $P$ is not faithful then $(0:a)\subset B$ by Lemma \ref{L:ann}. If $P$ is faithful  then $B\ne (0:(0:B))$ if $B=Pt$ for some $t\in R$. But, in this case $(0:a)$ is not of the form $Ps$ for some $s\in R$. So, $(0:a)\subset B$. By Lemma \ref{L:aFP} there exist $y\in F$ and $z\in K$ such that $x=ay+z$. It follows that $bx=bz\in bK$. Whence a contradiction.

{\bf 2.1}. $(0:B)=(0:b)\subset P$.

Let $r\in P\setminus (0:b)$. Then $(0:r)\subset Rb$. Let $0\ne c\in (0:r)$. There exists $t\in P$ such that $c=tb$. So, $(0:b)\subset (0:c)$. If $cx\in cK$, then $tbx=tby$ for some $y\in K$. Since $K$ is flat we get that $(bx-by)\in (0:t)K\subseteq bK$ ($c=bt\ne 0$ implies that $(0:t)\subset bR$), whence $bx\in bK$, a contradiction. Hence $cx\notin cK$. So, if we replace $b$ with $c$ we get the case {\bf 1}.

{\bf 2.2.1}. $(0:B)=(0:b)=P=bR$.

Since $b^2=0$ we have $b(bx)=0$. The flatness of $K$ implies that $bx\in bK$, a contradiction.

{\bf 2.2.2}. $(0 : B) = (0 : b) = P\ne bR$.

Let $c\in P\setminus bR$. Then $cx\notin K$ and there exists $s\in P$ such that $b=sc$. We have $(K : cx)=Rs\ne Rb$. So, if $scx\notin sK$ we get the case {\bf 2.1} by replacing $b$ with $s$ and $x$ with $cx$. Suppose that $scx\in sK$. We get $s(cx-y)=0$ for some $y\in K$. The flatness of F implies that $(cx-y)\in (0 : s)F\subset cF$. Whence $y=c(x +z)$ for some $z\in F$. If $y=cv$ for some $v\in K$ then $bx=sy=scv=bv\in bK$. This is false. Hence $c(x+z)\in K\setminus cK$ and $(0:(K:x+z))\subseteq (0:c)\subset Rs\subset P$. So, we get either the case {\bf 1} or the case {\bf 2.1} by replacing $x$ with $(x+z)$ and $b$ with $c$. 

Now, we assume that $R\ne Q$. First, we show that $G_Z$ is flat. Since $K$ and $F$ are flat then so are $K_Z$ and $F_Z$. If $Q$ is coherent, then $K_Z$ is FP-injective. So, it is a pure submodule of $F_Z$ and consequently $G_Z$ is flat. If $Q$ is not coherent, then $Z$ is flat, and by using \cite[Theorem 3]{Cou06} it is easy to show that $G_Z$ is almost FP-injective. Since $\mathrm{w.d.}(G_Z)\leq 1$, from above we deduce that $G_Z$ is flat. Let $G'$ be the kernel of the canonical homomorphism $G\rightarrow G_Z$. As in the proof of Proposition \ref{P:sharp} we have $G_Z\cong G/G'$. So, $G'$ is a pure submodule of $G$. For each $x\in G'$ there exists $s\in P\setminus Z$ such that $sx=0$. Hence $G'$ is a module over $R/Z$. But if $0\ne a\in Z$, $(0:a)\subseteq Z\subset (0:x)$ for any $x\in G'$. By Lemma \ref{L:aFP} $x=ay$ for some $y\in G$, and since $G'$ is a pure submodule, we may assume that $y\in G'$. Hence $G'=0$, $G\cong G_Z$ and $G$ is flat.  
\end{proof}

\begin{example}\label{E:IF}
Let $D$ be a valuation domain. Let $0\ne a\in P$. By \cite[Theorem 11]{Couch03} $R=D/aD$ is an IF-ring.
\end{example}

Now it is possible to state and to prove our main result.

\begin{theorem}
\label{T:main} For any almost FP-injective module $G$ over a chain ring $R$:
\begin{enumerate}
\item $\mathrm{w.d.}(G)=0$ if $R$ is an IF-ring;
\item if $R$ is a valuation domain which is not a field then:
\begin{enumerate}
\item $\mathrm{w.d.}(G)=0$ if $G_{\sharp}=0$ ($G$ is torsionfree);
\item $\mathrm{w.d.}(G)=1$ if $G_{\sharp}\ne 0$; 
\end{enumerate}
\item if $R$ is semicoherent but not coherent then;
\begin{enumerate}
\item $\mathrm{w.d.}(G)=\infty$ if $Z\subset G_{\sharp}$;
\item $\mathrm{w.d.}(G)=0$ if $G_{\sharp}\cup G^{\sharp}\subset Z$;
\item if $G_{\sharp}\cup G^{\sharp}=Z$ either $\mathrm{w.d.}(G)=0$ if $G$ is FP-injective or $\mathrm{w.d.}(G)=\infty$ if $G$ is not FP-injective;
\end{enumerate}
\item if $R$ is not semicoherent then $Z\otimes_RG$ is flat and almost FP-injective, and:
\begin{enumerate}
\item $\mathrm{w.d.}(G)=2$ if $Z\subset G_{\sharp}$;
\item $\mathrm{w.d.}(G)=0$ if $G_{\sharp}\cup G^{\sharp}\subset Z$;
\item if $G_{\sharp}\cup G^{\sharp}=Z$ either $\mathrm{w.d.}(G)=0$  or $\mathrm{w.d.}(G)=2$. More precisely $G$ is flat if and only if, for any $0\ne x\in G$, $(0:x)$ is not of the form $Qa$ for some $0\ne a\in Z$.
\end{enumerate}
\end{enumerate}
\end{theorem}
\begin{proof}
$(1)$. Since $R$ is IF then $G$ is FP-injective and flat.

$(2)$. It is an immediate consequence of the fact that $\mathrm{w.gl.d}(R)=1$.

$(3)$. From \cite[Theorem 2]{Cou12} which asserts that $\mathrm{f.w.d.}(R)=1$ in this case, we deduce that $\mathrm{w.d.}(G)<\infty\Rightarrow\mathrm{w.d.}(G)\leq 1$ and by Proposition \ref{P:non1} $\mathrm{w.d.}(G)=0$. We use Proposition \ref{P:sharp} to complete the proof. (Example \ref{E:example} shows there exist almost FP-injective modules $G$ which are not FP-injective, with $Z=G_{\sharp}\cup G^{\sharp}$).

$(4)$. First we assume that $G$ is injective. By Proposition \ref{P:inj} there exists a pure exact sequence:
$0\rightarrow K\rightarrow I\rightarrow G\rightarrow 0$, such
that $I$ is a direct sum of submodules isomorphic to $R$ or $H.$ If $I$ is flat then so is $G$. We may assume that $G$ is not flat. It follows that $\mathrm{w.d.}(I)=2$ by Proposition \ref{P:H}. By Proposition \ref{P:pure} $\mathrm{w.d.}(G)\leq 2$. Then, we assume that $G$ is FP-injective. So, $G$ is a pure submodule of its injective hull $E$. Again, by Proposition \ref{P:pure}, $\mathrm{w.d.}(G)\leq 2$. If $G$ is almost FP-injective then $G\cong E/D$ where $E$ is FP-injective and $D$ a pure submodule. We again use Proposition \ref{P:pure} to get $\mathrm{w.d.}(G)\leq 2$. We use Propositions \ref{P:sharp} and \ref{P:non1} to complete the proof of (a), (b) and the first assertion of (c). 

Let $0\ne x\in G$ and $A=(0:x)$. Since $G$ is a $Q$-module $A$ is an ideal of $Q$. Suppose that $A$ is not a non-zero principal ideal of $Q$ and let $0\ne r\in A$. Then $rQ\subset A$. It follows that $(0:A)\subset (0:r)$. Let $b\in (0:r)\setminus (0:A)$. Since $Q$ is not coherent there exists $a\in (0:r)\setminus Qb$. Then $(0:a)\subset (0:b)\subseteq A$ by Lemma \ref{L:ann}. By Lemma \ref{L:aFP} there exists $y\in G$ such that $x=ay$, whence $x\in (0:r)G$. Now suppose that $A=Qr$ for some $0\ne r\in Z$. If $a\in (0:r)$ then $rQ\subset (0:a)$, whence $x\notin aG$. So, $x\notin (0:r)G$. This completes the proof of (c).

It remains to prove the first assertion. By \cite[Theorem 2]{Cou12} $Z\otimes_ZG$ is flat. It is easy to check that $Z$ is a $Q$-module, whence so is $Z\otimes_RG$. Since $Q$ is self FP-injective we conclude that $Z\otimes_RG$ is almost FP-injective by Proposition \ref{P:flat}.
\end{proof}

We say that a chain ring is {\bf Archimedean} if its maximal ideal is the only non-zero prime ideal.

\begin{corollary}
\label{C:main}  For any almost FP-injective module $G$ over an Archimedean chain ring $R$:
\begin{enumerate}
\item $\mathrm{w.d.}(G)=0$ if $R$ is an IF-ring;
\item $\mathrm{w.d.}(G)\leq 1$ if $R$ is a valuation domain;
\item either $\mathrm{w.d.}(G)=2$ or $\mathrm{w.d.}(G)=0$ if $R$ is not coherent. More precisely $G$ is flat if and only if, for any $0\ne x\in G$ $(0:x)$ is not of the form $Ra$ for some $0\ne a\in P$. Moreover $P\otimes_RG$ is flat and almost FP-injective.
\end{enumerate}
\end{corollary}

Let us observe that $\mathrm{w.d.}(G)\leq 2$ for any almost FP-injective module $G$ over an Archime\-dean chain ring.

\begin{example}
Let $D$ be an Archimedean valuation domain. Let $0\ne a\in P$. Then $D/aD$ is Archimedean and it is IF by Example \ref{E:IF}.
\end{example}


\begin{corollary}
\label{T:proj} Let $R$ be an Archimedean chain ring. For any almost FP-injective $R$-module $G$ which is either countably generated or uniserial:
\begin{enumerate}
\item $\mathrm{p.d.}(G)\leq 1$ if $R$ is an IF-ring;
\item $\mathrm{p.d.}(G)\leq 2$ if $R$ is a valuation domain;
\item $\mathrm{p.d.}(G)\leq 3$ and $\mathrm{p.d.}(P\otimes_RG)\leq 1$ if $R$ is not coherent. 
\end{enumerate}
\end{corollary}
\begin{proof}
By \cite[Proposition 16]{Couc06} each uniserial module (ideal) is countably generated. So, this corollary is a consequence of Corollary \ref{C:main} and \cite[Lemmas 1 and 2]{Jens66}.
\end{proof}

\begin{remark}
If $R$ is an Archimedean valuation domain then $\mathrm{w.d.}(G)\leq 1$ and $\mathrm{p.d.}(G)\leq 2$ for any $R$-module $G$. 
\end{remark}


Let $R$ be a ring, $M$ an $R$-module. A $R$-module V is {\bf $M$-projective} if the natural homomorphism $\mathrm{Hom}_R(V ,M)\rightarrow\mathrm{Hom}_R (V , M /X)$ is surjective for every submodule $X$ of $M$. We say that  $V$ is {\bf quasi-projective} if $V$ is $V$-projective. A ring $R$ is said to be an {\bf fqp-ring} if every finitely generated ideal of $R$ is quasi-projective.

The following theorem can be proven by  using \cite[Lemmas 3.8, 3.12 and 4.5]{AbJaKa11}.
\begin{theorem}
\label{T:localfqp} \cite[Theorem 4.1]{Cou15}. Let $R$ a local ring and $N$ its nilradical. Then $R$ is a fqp-ring if and only if either $R$ is a chain ring or $R/N$ is a valuation domain and $N$ is a divisible torsionfree $R/N$-module.
\end{theorem}
\begin{lemma}
\label{L:simple} Let $R$ be a local ring and $P$ its maximal ideal. Assume that $(0:P)\ne 0$. If $R$ is a FP-injective module then $(0:P)$ is a simple module.
\end{lemma}
\begin{proof}
Let $0\ne a\in (0:P)$. By \cite[Corollary 2.5 ]{Ja73} $Ra=(0:(0:a))$. But $(0:a)=P$. So, $Ra=(0:P)$.
\end{proof}

\begin{example}
Let $D$ be a valuation domain which is not a field and $E$ a non-zero divisible torsionfree $D$-module which is not uniserial. Then $R=D\propto E$ the trivial extension of $D$ by $E$ is a local fqp-ring which is not a chain ring by \cite[Corollary 4.3(2)]{Cou15}.
\end{example}

\begin{theorem}
\label{T:main1} For any non-zero almost FP-injective module $G$ over a local fqp-ring $R$ which is not a chain ring, $\mathrm{w.d.}(G)=\infty$.
\end{theorem}
\begin{proof}
By \cite[Proposition 5.2]{Cou15} $\mathrm{f.w.d.}(R)\leq 1$. If $N$ is the nilradical of $R$ then its quotient ring $Q$ is $R_N$ and this ring is primary. 

First assume that $R=Q$. Then each flat module is free and $\mathrm{f.w.d.}(R)=0$ by \cite[Theorems P and 6.3]{Bas60}. If $G$ is free, it follows that $R$ is FP-injective. By Lemma \ref{L:simple} this is possible only if $N$ is simple, whence $R$ is a chain ring. Hence $\mathrm{w.d.}(G)=\infty$.

Now assume that $R\ne Q$ and $\mathrm{w.d.}(G)\leq 1$. Then $\mathrm{w.d.}(G_N)\leq 1$. Since $\mathrm{f.w.d.}(Q)=0$ it follows that $G_N$ is flat. As in the proof of Proposition \ref{P:non1}, with the same notations we show that $G_N\cong G/G'$ and $G'$ is a module over $R/N$. But if $0\ne a\in N$, $(0:a)=N\subset (0:x)$ for any $x\in G'$. As in the proof of Lemma \ref{L:aFP} we show that $x=ay$ for some $y\in G$. Since $G'$ is a pure submodule, we may assume that $y\in G'$. Hence $G'=0$, $G\cong G_N$ and $G$ is flat. We use the first part of the proof to conclude.
\end{proof}

\end{document}